\documentclass[11pt, reqno]{amsart}

\usepackage{amsmath}
\usepackage{amsthm}
\usepackage{amssymb}
\usepackage{amsfonts}
\usepackage[mathscr]{eucal}
\usepackage{latexsym}
\usepackage{graphicx}
\usepackage{xcolor}
\usepackage{hyperref}
\usepackage{enumerate}

\newtheorem{prop}{Proposition}[section]
\newtheorem{cor}[prop]{Corollary}
\newtheorem{thm}[prop]{Theorem}
\newtheorem{lem}[prop]{Lemma}

\theoremstyle{definition}

\newtheorem{example}[prop]{Example}

\renewcommand{\phi}{\varphi}
\renewcommand{\Re}[1]{\operatorname{Re} #1}

\newcommand{\C}{\mathbb{C}}

\newcommand{\R}{\mathbb{R}}

\DeclareTextFontCommand{\textnf}{\normalfont}

\newcommand{\ol}{\overline}

\begin{document}

\title{Positive functionals and Hessenberg matrices}

\author{Jean B. Lasserre}
\address[J-B.~Lasserre]{LAAS-CNRS and Institute of Mathematics, University of Toulouse, France}
\email{\tt lasserre@laas.fr}

\author{Mihai Putinar}
\address[M.~Putinar]{University of California at Santa Barbara, CA,
USA and Newcastle University, Newcastle upon Tyne, UK} 
\email{\tt mputinar@math.ucsb.edu, mihai.putinar@ncl.ac.uk}

\date{\today}


\subjclass[2010]{}

\begin{abstract} Not every positive functional defined on bi-variate polynomials of a prescribed degree bound is represented by the integration against a positive measure.
We isolate a couple of conditions filling this gap, either by restricting the class of polynomials to harmonic ones, or imposing the vanishing of a defect indicator. Both criteria offer
constructive cubature formulas and they are obtained via well known matrix analysis techniques involving either the dilation of a contractive matrix to a unitary one or the specific structure of the Hessenberg matrix associated to the multiplier by the underlying complex variable.

\end{abstract}
\maketitle

\section{Introduction} Typically, a positive functional on a space of continuous functions is represented as the integral
against a positive measure. The departure from this paradigm is notable on finite dimensional spaces, turning the existence of a
 Riesz representation theorem into the basic quest of numerical cubature, solutions of the truncated moment problem or algebraic
 certificates for positive elements. 
 
 The present note analyses the structure of positive functionals on spaces of bi-variate real polynomials of a prescribed degree. In spite of the ample 
 references and recent progress on this topic, there are still challenging open questions and notorious difficulties, mostly related to the constructive aspects.
 Our approach uses complex variables, hence Hermitian forms over complex variables rather than quadratic forms over the real field. The natural Hilbert space realization of Hermitian forms is
 also prominent in our note. The balance between Hermitian algebra and Hilbert space geometry turns out to be beneficial for entering into the fine structure of positive functional on polynomial spaces.

 Specifically, we import from operator theory the well known Sz.-Nagy unitary dilation of a contractive matrix for constructing numerical cubatures for {\it harmonic} polynomials; second,
 we identify, via a normality criterion in matrix analysis, the numerical obstruction of a positive functional on polynomials bi-variate polynomials of a prescribed degree to possess a cubature on a
 real subspace of roughly half-dimension, but certainly containing all polynomials of half-degree. In this quest we naturally touch the structure of the Hessenberg matrix representing the multiplier by the complex variable, and reveal its normality up to a rank-one additive perturbation. A normal Hessenberg matrix characterizes finite point cubatures for polynomials of a fixed maximum degree in the two variables. For the reader versed in some operator theory, this part of our note deals with the Krylov subspace analysis of truncations of a hyponormal operator.
 
 A unifying conclusion of our observations is that one has to pay a price for a positive functional to admit numerical quadratures (with positive weights). Namely one has to drop the expectations/identities to some distinguished subspaces of polynomials of approximatively half the original dimension. Undoubtedly there are other paths to achieve such cubature formulas on privileged polynomials, such it is the case of quadrature domains for harmonic functions. A common feature of our study is the numerical and computational accessibility to cubatures, modulo ubiqutiuous matrix analysis techniques, having as solely input data the moments of the original functional.
 
  While a selected group of operator theory experts may find the main results below not surprising, our aim is to disseminate to a larger audience some possibly novel, but very basic and versatile matrix analysis techniques. The recent authoritative monograph by Schm\"udgen \cite{Sch} contains ample details and references, old and new, on the structure of positive functionals on polynomial subspaces, cubature formulas, orthogonal polynomials and (truncated) moment problems. All presented from the traditional point of view of the functional analyst or function theorist. Hessenberg matrices however are cultivated by rather disjoint groups of mathematicians, notably in numerical analysis \cite{LS} and approximation theory \cite{Simanek}.
 \bigskip
 
 {\bf Acknowledgement.} 
Research of the first author and visit of the second  author at LAAS in Toulouse, funded by the European Research Council (ERC) under the European Union's Horizon 2020 research and innovation program (grant agreement ERC-ADG 666981 TAMING).

\section{Preliminaries}

We denote by $z = x+iy$ the complex variable in $\C$, identified with $\R^2$.
For a fixed positive integer $d$, $\C_d[z]$ is the complex vector space of polynomials
of degree less than or equal to $d$, and similarly for $\C_d[z,\ol{z}]$, respectively $\C_d[x,y]$
where in the case of two variables the total degree is adopted. Similarly $\C_{k, \ell}[z,\ol{z}]$
stands for the vector space of polynomials of bidegree $k$ in $z$, respectively $\ell$ in $\ol{z}$.

The starting point is a linear functional
$$ L : \C_{2d}[z,\ol{z}] \longrightarrow \C$$
which is real $L(\ol{p}) = \ol{L(p)}, \ p  \in  \C_{2d}[z,\ol{z}] ,$
and {\it positive semidefinite}, that is
$$ L( |f(z,\ol{z})|^2) \geq 0,$$ for all $f  \in  \C_{d}[z,\ol{z}] $.
In particular, but not equivalently,  $L$ is {\it hermitian positive semidefinite}, meaning that
$$ L( |f(z)|^2) \geq 0,\  f  \in  \C_{d}[z] .$$
A {\it hermitian square} is by definition a real polynomial of the form $|p(z)|^2$ with $p \in \C[z]$.

The inner product and associated semi-norm 
$$ \langle p, q \rangle_L = L(p \ol{q}), \ \ \| p \|_L^2 = L(|p|^2),$$
are considered either on $ \C_{d}[z,\ol{z}] $ or on the subspace $\C_d [z]$.
In either case, Cauchy-Schwarz inequality holds
$$ |\langle p, q \rangle_L| \leq \| p \|_L \|q \|_L,$$
as well as triangle inequality.
An immediate consequence of the latter is that the set of null vectors
$$ N = \{ p \in C_{d}[z,\ol{z}], \ \|p\|_L = 0\},$$
is a vector subspace of $C_{d}[z,\ol{z}]$. The quotient space
$H_d = \C_{d}[z,\ol{z}]/N$ endowed with the induced norm $\| \cdot \|_L$ is therefore a Hilbert space of finite dimension.

For any element $f \in \C_{d-2}[z,\ol{z}] $ Cauchy-Schwarz inequality implies 
$$ \| z f \|_L^2 = L( |z|^2 |f|^2) \leq \| |z|^2 f \|_L \|f \|_L.$$ In particular, if
$ \| f \|_L = 0,$ then $\|z f \|_L = 0$. Thus multiplication by the complex variable 
is well defined as an induced linear transformation
$$ M_z : \C_{d-2}[z,\ol{z}]/(N \cap \C_{d-2}[z,\ol{z}]) \longrightarrow \C_{d-1}[z,\ol{z}]/(N \cap \C_{d-1}[z,\ol{z}]).$$
The above quotient spaces can be interpreted thanks to the second isomorphism theorem as subspaces of $H_d$.

Similarly we can define the subspace $A_d = C_{d}[z]/(C_{d}[z] \cap N) $ and speak about the well defined operator
$M_z : A_{d-2} \longrightarrow A_{d-1}$. Denote by $\pi_{d-2}$ the orthogonal projection of $A_d$ (or even $H_d$) onto
$A_{d-2}$. When there is no danger of confusion we simply denote $\pi = \pi_{d-2}$. The compressed operator
$$ M = M_{d-2} = \pi M_z|_{A_{d-2}}$$
is therefore well defined and bounded, as an endomorphism of a finite dimensional Hilbert space.

\section{Quadrature formulas}  We start by isolating a direct consequence of Sz-Nagy dilation theorem.

\begin{thm}\label{harmonic} Let $L : \C_{2d+2}[z,\ol{z}] \longrightarrow \C$ be a real functional which is non negative on hermitian squares, and
let $R = \| M_{d}  \|,$ with the notation adopted in the preliminaries. There are at most $N \leq (d+1)^2$ nodes $z_k = R e^{i \theta_k}$ on the circle with center at $z=0$ and radius $R$,
and positive weights $c_k >0, 1 \leq k \leq N$, such that
$$ L(h) = \sum_{k=1}^N c_k h(z_k, \ol{z}_k),$$
for all harmonic polynomials $h  \in \C_{d,d}[z,\ol{z}].$
\end{thm}

\begin{proof} The linear transform $T= M_{d}/R$ is contractive on the space $A_{d}$ of complex polynomials. Counting monomials as generators we find that $\dim A_{d} \leq d+1$.
By the finite dimensional analog of Sz.-Nagy dilation theorem, there exists a Hilbert space $K$ of dimension less than or equal to $(d+1)^2$, containing isometrically $A_d$, and a
unitary transformation $U: K \longrightarrow K$ with the property:
$$ \langle T^k {\bf 1}, {\bf 1} \rangle _L = \langle U^k {\bf 1}, {\bf 1}\rangle_K, \ \ 0 \leq k \leq d.$$
Throughout this note ${\bf 1}$ stands for the constant function equal to $1$.
For a proof see Sz.-Nagy appendix to the Functional Analysis treatise \cite{RN}, or the original construction in \cite{E}.
The spectral resolution of the unitary matrix $U$:
$$ U = \sum_{k=1}^{(d+1)^2} e^{i \theta_k} \langle \cdot, f_k\rangle f_k,$$
implies for any element $p\in \C_d[z]$:
$$ L(p) = \langle p(T) {\bf 1}, {\bf 1} \rangle _L = \langle p(U) {\bf 1}, {\bf 1}\rangle_K = $$
$$ \sum_k p(e^{i \theta_k}) |\langle {\bf 1}, f_k\rangle|^2.$$
Above $f_k$ are the unit, mutually orthogonal eigenvectors of $U$.

Since every real valued harmonic polynomial $h  \in \C_{d,d}[z,\ol{z}]$ can be represented as the real part of a complex polynomial of degree $d$ or less, the proof is complete.
The bound $N \leq (d+1)^2$ counts the non-zero weights $c_k =  |\langle {\bf 1}, f_k\rangle|^2.$

\end{proof} 

It is relevant for applications to remind that the construction of the larger unitary matrix $U$ is explicit, on blocks of size $(d+1)\times (d+1)$ involving at most two square roots of
positive matrices, to be more specific, with the notation in the proof, the {\it defect matrices} $\sqrt{I-T^\ast T}$ and $\sqrt{I-TT^\ast}$, cf. \cite{E}.

Note that in the proof of the preceding theorem we have used only the hermitian positivity of the linear form $L$. 
If we assume $L$ positive semi-definite, then we may expect more. We elaborate two such consequences.

First, remark that a real polynomial change of variables $X = X(z,\ol{z}), Y = Y(z,\ol{z})$ (specifically $X = \ol{X}$ and $Y = \ol{Y}$) enables to consider the subalgebra $\C[X,Y]$ and the harmonic polynomials
there. That is pull-backs by the base change of harmonic polynomials in the variables $X,Y$. Since we deal with bounded degrees we have to impose the necessary degree reductions. Assume that
$\max( \deg X, \deg Y) \leq n$. Then the functional $L$ defines by pull-back a non-negative functional on $C_{2e+2} [X,Y]$ with $2e+2$ the largest even integer less than or equal to $2d/n$. 
The corresponding ``complex" variable is $Z = X(x,y) + i Y(x,y)$ and so on. We leave the details 
of adapting Theorem \ref{harmonic} to the reader.

As simple as it might be, the following observation is instrumental for the rest of the article.

\begin{lem}\label{compression}  Assume the real functional $L : \C_{2d+2}[z,\ol{z}] \longrightarrow \C$ is non-negative on hermitian squares and denote
$M = \pi_d M_z \pi_d$.Then
\begin{equation}
L(p\ol{q}) = \langle p(M){\bf 1}, q(M) {\bf 1}\rangle_L,
\end{equation}
provided $\max(\deg(p), \deg(q)) \leq d+1, \deg(p\ol{q}) \leq 2d+1.$
\end{lem}

For the proof note that 
$$ M p(M){\bf 1} = M p(z) = \pi_{d} (z p(z))$$
whenever $\deg(p) \leq d.$ In short we do not allow in the Lemma both polynomials $p$ and $q$ to have maximal degree $d+1$.

Second, and more interesting, are some additional constraints on the matrix 
$M = M_{d}$, assuming that the functional $L$ is defined on $\C_{2d+2}[z,\ol{z}]$
and it is non-negative on all real squares.

Let $f(z,\ol{z})$ be a polynomial of degree at most $d$ in $z$, respectively $\ol{z}$:
$$ f(z, \ol{z}) = \sum_{j,k=0}^{d} f_{jk} z^j \ol{z}^k.$$ 
The positivity of $L$ implies in view of Lemma \ref{compression}
$$ 0 \leq L(|f|^2) = \sum_{j,k,r,s} f_{jk} \ol{f_{rs}} L(z^{j+s}\ol{z}^{k+r}) = \sum_{j,k,r,s} f_{jk} \ol{f_{rs}} \langle M^{j+s} {\mathbf 1}, M^{k+r} {\mathbf 1}\rangle.$$
Definitely these are non-trivial constraints on the matrix $M = M_{d}$. If we simplify the form of $f$ to
$$ g(z,\ol{z}) = p(z) + \ol{z} q(z),$$
 Lemma \ref{compression} yields
$$ 0 \leq L(|g|^2) = L(|p(z)|^2 + |z|^2 |q(z)|^2) + 2\Re L(\ol{z} \ol{p(z)}q(z)) =$$
$$ \| p \|^2 + \| M q \|^2 + 2\Re \langle Mp, q\rangle$$
whenever $\deg(p) \leq d$ and $\deg(q) \leq d-1$.
This is reminiscent of Halmos-Bram subnormality condition for the matrix $M$,
except that the quadratic form, defined on $A_d \oplus A_d$:
$$ \sigma_M(p,q) = \| p \|^2 + \| M q \|^2 + 2\Re \langle Mp, q\rangle$$
is positive semi-definite only on the codimension-one subspace $A_d \oplus A_{d-1}$.
For details about the operator theory background see for instance Halmos' problem book \cite{Hal}.
Consequences of such positivity conditions, and in particular relations to moment problems are explicitly
discussed in the monograph \cite{MP} and the article \cite{CP}, where a question raised by Halmos was solved.

If we impose the stronger condition
\begin{equation}\label{hypo}
 \| p \|^2 + \| M q \|^2 + 2 \Re \langle M p, q \rangle \geq 0, \ \ p, q \in A_{d},
 \end{equation}
then the matrix $M$ is what is called hyponormal, that is $[M^\ast,M] \geq 0$. Since the trace of the commutator of two matrices is always zero, we infer that
$M$ is actually normal. The spectral resolution for normal matrices will then provide the desired quadrature formula.

\begin{thm}\label{real-pos} Let $L : \C_{2d+2}[z,\ol{z}] \longrightarrow \C$ be a real functional which is non negative on all real squares.
Assume that $L$ is strictly positive on hermitian squares of bi-degree less than or equal to $(d,d)$.

There are at most $d+1$ points $a_k \in \C$ and weights $c_k>0$, so that
\begin{equation}\label{qd}
 L(p(z) \ol{q(z)}) = \sum_{k=1}^N c_k p(a_k) \overline{q(a_k)} 
 \end{equation}
for all $p,q \in \C_{d+1}[z]$, subject to $\deg(p\overline{q}) \leq 2d+1$
if and only if the only non-positive eigenvalue of the quadratic form $\sigma_M$ is equal to zero.
\end{thm} 

We have stated the conclusion in terms of one defect number, to stress that a single numerical criterion
certifies the existence of the quadrature (\ref{qd}). 

\begin{proof} Assume condition (\ref{hypo}) holds. Then 
$$ |\langle Mp, q \rangle| \leq \| p \| \|Mq\|, \ \ p, q \in A_d.$$
On the other hand
$$ |\langle Mp, q\rangle| = |\langle p, M^\ast q\rangle| \leq \|p \| \|M^\ast q\|, \   p, q \in A_d,$$
and the inequality is an equality for $p = M^\ast q$.
Hence
$$ \|M^\ast q\|^2 \leq \|M^\ast q\| \| M q\|, \ q \in A_d,$$ 
and consequently
$$ \|M^\ast q \| \leq \| M q\|$$
regardless $M^\ast q = 0$ or not.
But the later inequality means
$$ \langle M M^\ast q, q \rangle \leq \langle M^\ast M q, q \rangle, \ \ q \in A_d,$$
or in operator inequality terms, the commutator $[M^\ast, M]$ is non-negative.

Since ${\rm trace} [M^\ast, M] =0$ we deduce that $M$ is a normal matrix, with a spectral decomposition
$$ M = \sum_{k=0}^n a_k \langle \cdot, f_k\rangle f_k,$$ 
with $n \leq d+1$. Lemma \ref{compression} implies
$$  L(p(z) \ol{q(z)}) = \langle p(M){\bf 1}, q(M){\bf 1}\rangle = \sum_{k=1}^N p(a_k) \overline{q(a_k)} |\langle {\bf 1}, f_k\rangle|^2$$
for all  $p \in \C_{d+1}[z]$ and $q \in \C_{d+1}[z]$, but excluding the top degree corner, that is so that $\deg(p\overline{q}) \leq 2d+1$.

Conversely, assume that quadrature formula \ref{qd} holds in the given range of bi-degrees. Then we can reverse the preceding computations and find
$$ \langle p(M){\bf 1}, q(M){\bf 1}\rangle = \sum_{k=1}^N c_k p(a_k) \overline{q(a_k)}, $$
whenever $p, q \in \C_{d+1}[z]$ and $\deg(p\overline{q}) \leq 2d+1$. The assumption that the functional $L$
is stricly positive on hermitian squares of bidegree less than or equal to $(d,d)$ is equivalent to the fact that the vectors
${\bf 1}, z = M{\bf 1}, \ldots, z^d = M^d {\bf 1}$ form a basis of the vector space $A_d$. In other terms $M$ is a matrix with a cyclic vector, acting
on a $(d+1)$-dimensional space. Its minimal polynomial $h(\lambda)$ coincides then with the characteristic polynomial $\det(\lambda I - M)$, and in particular
has degree $d+1$. As a matter of fact one can identify $h(\lambda)$ with the orthogonal polynomial in degree $d+1$ of the associated inner product.
Whence
$$ 0 = \langle h(M){\bf 1}, q(M){\bf 1}\rangle = \sum_{k=1}^N c_k h(a_k) \overline{q(a_k)}.$$
By assumption, $N \leq d+1$, so that we can choose the polynomial $q(z)$ of degree at most $d$ to be zero at all points $a_k$ except one.
Consequently $h(a_k) = 0$ for all $k, 1 \leq k \leq N$.

Let $P,Q \in \C[z]$ be arbitrary degree polynomials. The division algorithm yields
$$ P = P_0 h + P_1, \ \ Q = Q_0 h + Q_1,$$
with $\max(\deg(P_1),\deg(Q_1)) \leq d$. Since $P_0(M) h(M) =0$ and $P_0(a_k) h(a_k)=0, 1 \leq k \leq N,$ and similarly for $Q$, we find
the identity
$$ \langle P(M){\bf 1}, Q(M){\bf 1}\rangle = \sum_{k=1}^N c_k P(a_k) \overline{Q(a_k)}$$
valid over the entire polynomial ring $\C[z]$. In other terms the operator $M$ is unitarily equivalent to the multiplication
by the complex variable on Lebesgue space $L^2(\nu)$, with $\nu = \sum_{k=1}^N c_k \delta_{a_k}.$
That is $M$ is normal and the quadratic form $\sigma_M$ is necessarily positive semi-definite.
\end{proof}

A few remarks are in order. One can start with the functional $L$ merely non-negative on hermitian squares, but then
the statement has be adapted to include the positivity of the form $\sigma_M$, and not only the reference to its only possible negative square.
The following section gives more details along this path.
Given the non-degeneracy assumption in the statement of Theorem \ref{real-pos}, we found during the proof that the number $N$ of nodes
is necessarily maximal, that is $N = d+1$. Finally one can compress the multiplier $M_z$ to other subspaces and have cubature formulas derived
from the same normality principle. Some examples collected in the final section will complement these general statements.

\section{The Hessenberg matrix representation} We keep the notation introduced in the previous section and assume that the functional $L : \C_{2d+4}[z,\ol{z}] \longrightarrow \C$ is real and non negative on all real squares.
We adopt the hypothesis of Theorem \ref{real-pos}, namely that $L$ is strictly positive on hermitian squares of bi-degree less than or equal to $(d,d)$. Then one can speak without ambiguity of the associated orthogonal complex polynomials $P_j(z), \ 0 \leq j \leq d+1:$
$$ L(P_j \ol{P_k}) = \delta_{jk}, \ \ 0 \leq j,k \leq d+1.$$
We can also assume that the leading term of $P_j$ is positive:
$$ P_j(z) = \kappa_j z^j + O(z^{j-1}), \ \kappa_j >0, \ \ \ 0 \leq j \leq d+1.$$
In particular the cyclic vector ${\bf 1}$ has the coordinates:
$$ {\bf 1} = (\kappa_0, 0, 0, \ldots, 0)^T.$$
The $(d+1) \times (d+1)$ matrix representation of the compressed multiplier $M = \pi_d M_z \pi_d : \C_d[z] \longrightarrow \C_d[z]$ with respect to the orthonormal basis
$\{ P_0, P_1, \ldots, P_d\}$ has only a first sub-diagonal non-zero:
$$ a_{jk} := \langle M P_k, P_j \rangle = 0, \ \ k < j-1.$$ Such a structure bears the name of a {\it Hessenberg matrix}:
$$ M = \left( \begin{array}{cccccc}
               a_{00}&a_{01}&a_{02}& \ldots& a_{0,d-1}&a_{0d}\\
               a_{10}&a_{11}&a_{12}&\ldots& a_{1,d-1}&a_{1d}\\
               0&a_{21}&a_{22}&\ldots&a_{2,d-1}&a_{2d}\\
               0&0&a_{32}&\ldots& a_{3,d-1}&a_{3d}\\
               \vdots & & \ddots &\ddots & & \vdots\\
               0&0& \ldots & & a_{d,d-1}&a_{dd}\\
               \end{array} \right).$$
Note that the sub-diagonal entries are non-zero:
$$ a_{j+1,j} = \langle M P_j, P_{j+1} \rangle = \langle \kappa_j z^{j+1} + \ldots, P_{j+1} \rangle = \frac{\kappa_j}{\kappa_{j+1}}, \ \ 0 \leq j <d.$$

Returning to the quadratic form $\sigma_M$ defined in the previous section, we turn to its matrix representation:
$$ \sigma_M(p,q) = (\ol{p}, \ol{q}) \left( \begin{array}{cc}
                        I&M^\ast\\
                        M& M^\ast M\\
                        \end{array}\right) \left( \begin{array}{c}
                        p\\
                        q
                        \end{array}\right).$$
         Our assumption implies, as explained in the previous section, that the above block matrix is non-negative on the subspace $C_{d}[z] \oplus C_{d-1}[z]$.  
         
         A notable change of coordinates block-diagonalizes this matrix with self-commutator $[M^\ast, M]$ as one of the blocks:
         $$  \left( \begin{array}{cc}
                        I&0\\
                        -M& I\\
                        \end{array}\right)    \left( \begin{array}{cc}
                        I&M^\ast\\
                        M& M^\ast M\\
                        \end{array}\right)  \left( \begin{array}{cc}
                        I&-M^\ast\\
                        0& I\\
                        \end{array}\right) = \left( \begin{array}{cc}
                        I&0\\
                        0& [M^\ast, M]\\
                        \end{array}\right).$$
                        In virtue of the min-max principle, the self-adjoint matrix $[M^\ast,M]$ has at most one negative eigenvalue.
                        
With some elementary computations, we can say more about the possible negative eigenspace of this self-commutator. To be more precise, start with the one dimension up compression
  $M_{d+1} = \pi_{d+1} M_z \pi_{d+1}$. This is well defined as a linear transform due to the assumption that the original functional is non-negative on $\C_{2d+4}[z,\ol{z}]$. All the above computations hold, except the existence of the orthogonal polynomial $P_{d+2}$, which is not needed in what follows.
  In particular the matrix   $$  \left( \begin{array}{cc}
                        I&M_{d+1}^\ast\\
                        M_{d+1} & M_{d+1} ^\ast M_{d+1}\\
                        \end{array}\right)   $$
                        is non-negative on the subspace $\C_d[z] \oplus \C_d[z]$. 
                        That is, the matrix
                        $$ \left( \begin{array}{cc}
                        I&M^\ast\\
                        M& \pi_d M_{d+1} ^\ast M_{d+1}\pi_d \\
                        \end{array}\right)$$
                        is non-negative on   $\C_d[z] \oplus \C_d[z].$
   But   $$ \pi_d M_{d+1} ^\ast M_{d+1}\pi_d   = M^\ast M + \pi_d M_{d+1} ^\ast (\pi_{d+1}-\pi_d) M_{d+1}\pi_d.$$
   It remains to identify along the orthonormal basis given by the complex polynomials $P_j$ the contribution of the last rank-one matrix:
   $$  \pi_d M_{d+1} ^\ast (\pi_{d+1}-\pi_d) M_{d+1} P_j = \pi_d M_{d+1} ^\ast (\pi_{d+1}-\pi_d) (z P_j (z)) = $$ $$\pi_d M_{d+1} ^\ast (  z P_j (z) - z P_j (z)) = 0, \  j < d,$$
   and                        
$$  \pi_d M_{d+1} ^\ast (\pi_{d+1}-\pi_d) M_{d+1} P_d =  \pi_d M_{d+1}^\ast \langle z P_d, P_{d+1}\rangle P_{d+1} = a_{d, d+1} \pi_d M_{d+1}^\ast P_{d+1}.$$
Note that $K = \pi_d M_{d+1} ^\ast (\pi_{d+1}-\pi_d) M_{d+1}\pi_d$ is a non-negative self-adjoint matrix, with the only non-zero entry on the $d \times d$ diagonal position:
$$ \langle K q, q \rangle = a_{d+1,d}^2 |\langle q, P_d \rangle|^2, \ \ q \in \C_d[z].$$
We infer that the matrix
$$ \left( \begin{array}{cc}
                        I&M^\ast\\
                        M& M^\ast M + K\\
                        \end{array}\right)$$
                        is non-negative on the whole space $\C_d[z]\oplus \C_d[z]$, and via the same elementary transforms we conclude that
                        $[M^\ast,M] + K$ is a non-negative matrix.
  This information gives a lower bound for the only possible negative eigenvalue of the self-commutator:
  $$ \lambda_{-} = \inf \{ \langle [M^\ast,M] q, q \rangle; \ \ \| q \| \leq 1 \}. $$ Indeed, fix a polynomial $q \in \C_d[z]$ of norm less than $1$. Then
  $$ \lambda_{-} + a_{d+1,d}^2 \geq \langle [M^\ast,M] q,q \rangle +    a_{d+1,d}^2   \|q \|^2 \geq    $$ $$
  \langle [M^\ast,M] q,q \rangle + a_{d+1,d}^2 |\langle q, P_d \rangle|^2 = \langle (  [M^\ast,M] + K)q, q \rangle \geq 0.$$       
  
  We collect the above computations into a single statement.
  
  \begin{thm}\label{Hess} Let $L : \C_{2d+4}[z,\ol{z}] \longrightarrow \C$ be a real functional which is non negative on all real squares.
Assume that $L$ is strictly positive on hermitian squares of bi-degree less than or equal to $(d,d)$. Denote by $(a_{jk})_{j,k=0}^{d+1}$ the Hessenberg matrix associated
to the system of complex orthogonal polynomials induced by $L$.

Then the matrix $M = (a_{jk})_{j,k=0}^{d}$ is normal after a rank-one perturbation, in particular the self-commutator $[M^\ast,M]$ has at most one negative eigenvalue $\lambda_{-} \geq - a_{d+1,d}^2.$
Moreover, the matrix $M$ is normal, that is $[M^\ast,M] =0$ if and only if $a_{d+1,d}=0$.
\end{thm} 

Remark that the non-negativity assumption in the statement implies that the $(d+2) \times (d+2)$ Hessenberg matrix $\tilde{M} = (a_{jk})_{j,k=0}^{d+1}$ is well defined, hence the normality of its
submatrix $M$ is equivalent to the fact that $\tilde{M}$ leaves invariant the subspace $\C_d[z]$. More precisely,
$$ \tilde{M} = \left( \begin{array}{cc}
                           M&\ast\\
                           u& a_{d+1,d+1}\\
                           \end{array} \right), \ \ {\rm where} \ \ \ u = (0,0,\ldots, 0, a_{d, d+1}).$$
While there is a simpler, operator theoretic derivation of the conclusion of the above theorem, we preferred to enter minimally into the structure of complex orthogonal polynomials for alerting the 
reader about the potential of the Hessenberg matrix approach. For the linear numerical analyst or approximation theory expert it is hardly a surprise to put the Hessenberg matrix representation at work,
see for instance \cite{HJ,Simanek}, to only touch the ample literature devoted to the subject.

With the above understanding of the normality of the compressed multiplier $M$, we can return to the main result of the previous section, with a quantitative criterion.

\begin{cor} Let $L : \C_{2d+4}[z,\ol{z}] \longrightarrow \C$ be a real functional which is non negative on all real squares.
Assume that $L$ is strictly positive on hermitian squares of bi-degree less than or equal to $(d,d)$.

There are at most $d+1$ points $a_k \in \C$ and weights $c_k>0$, so that
$$
 L(p(z) \ol{q(z)}) = \sum_{k=1}^N c_k p(a_k) \overline{q(a_k)} 
$$
for all $p,q \in \C_{d+1}[z]$, subject to $\deg(p\overline{q}) \leq 2d+1$
if and only if the Hessenberg matrix entry $a_{d, d+1}$ vanishes.
\end{cor}
  
  Summing up, the main observation supporting the preceding Corollary is stated in the following rather surprising property of the associated Hessenberg matrix.
  
  \begin{prop} Let $L : \C_{2d+4}[z,\ol{z}] \longrightarrow \C$ be a real functional which is non negative on all real squares.
Assume that $L$ is strictly positive on hermitian squares of bi-degree less than or equal to $(d,d)$ and denote by $M_k = \pi_k M_z \pi_k$
the associated Hessenberg matrix of order $k, \  k \leq d+1$.

The following are equivalent:

\begin{enumerate}
\item $M_d$ is a normal matrix,

\item $\det ([M_d^\ast,M_d]+ \epsilon I) \geq 0$ for $\epsilon >0$,

\item The entry $a_{d, d+1}$ in $M_{d+1}$ vanishes,

\item$M_{d+1}$ leaves invariant the subspace $\C_d[z]$.

\end{enumerate} 

\end{prop}

To complete the proof, it suffices to recall that the self-commutator $[M^\ast,M]$ can possess at most one negative eigenvalue. Condition $(2)$ simply states that $[M^\ast,M]$ cannot 
have a negative eigenvalue, hence $[M^\ast,M] \geq 0$, hence $M$ is normal. We stress that in general condition (2) offers a numerical certificate for a symmetric matrix with at most one 
negative eigenvalue to be non-negative.

So far, we have only used the positivity of the functional $L$ on hermitian squares and on squares of the form $|f(z) + \ol{z} g(z)|^2$. A whole hierarchy of stronger non-equivalent conditions, for instance on the
Hessenberg matrix $M$, can be derived from the positivity of $L$ on elements $|P(z,\ol{z})|^2$ where $\deg_{\ol{z}}P \leq k$. This later property, called $k$-hyponormality, is analyzed in detail and put at work in \cite{CP}.

\section{Examples}

In this section we gather a few simple examples which support the general theme of this note: that combining complex variables with positivity of linear functionals 
is leading to useful, sometimes unexpected, consequences.

\begin{example} Contrary to the tacit convention of this note, we start with a scalar product rather than a potential integration functional with positivity
assumption. Specifically, let $[-a,a]$ be a compact interval of the real line and consider the Dirichlet type inner product:
$$ \langle p, q \rangle = \int_{-a}^a [p(x)\overline{q(x)} + p'(x)\overline{q'(x)}] dx,$$
defined for all polynomials $p,q \in \C[x]$ with complex coefficients. Apparently there is no linear functional $L$, so that
$$ L(p\overline{q}) = \langle p, q \rangle.$$
One can remedy this by passing to complex coordinates. Namely, define for $p, q \in \C[z]$ the sesquilinear form
$$ \Lambda(p\overline{q}) = \int (1+\frac{1}{4}\Delta)(p\overline{q}) d\mu,$$
where $\mu$ is any rapidly decreasing at infinity positive measure. Above $\Delta$ is Laplace operator, and its factorization via complex variables
$$ \Delta = 4 \frac{\partial}{\partial z} \frac{\partial}{\partial \ol{z}}$$ implies
$$  \Lambda(p\overline{q}) = \int (p\overline{q} + p'\overline{q'}) d\mu,$$
when this time $p'(z), q'(z)$ denote complex derivatives. By choosing $d\mu = \chi_{[-a,a]} dx$ one finds unexpectedly
$$ \Lambda(p\ol{q}) =  \langle p, q \rangle, \ \ p,q \in \C[z].$$
The multiplier $M_z$ is well defined on the whole algebra $\C[z]$, but its compressions to prescribed degrees are not normal.
However, a familiar computation borrowed from Sturm-Liouville theory shows that one can restrict $M$ to a space of polynomials where it is. To be precise,
$$ \langle M_z p(z), q(z) \rangle + \langle p(z), M_z q(z) \rangle = $$ $$ \int_{-a}^a [ xp(x)\ol{q(x)} + p(x) \ol{xq(x)} + (xp'(x) + p(x))\ol{q'(x)} +p'(x) \ol{(x q'(x) + q(x))}] dx=
$$ $$ 2 \int_{-a}^a (1+\frac{1}{4}\Delta)(p\overline{q}) x dx + \int_{-a}^a [p(x)\ol{q'(x)} +p'(x) \ol{q(x)}] dx.$$
Fix a dimension $d$ and consider the subspace $V(e)_d$ of even polynomials vanishing at $\pm a$, of degree less or equal than $d$, or the subspace $V(o)_d$ of odd polynomials
vanishing at $\pm a$. Along every one of these two subspaces the last to integrals vanish. Hence the compressions $M_e$, respectively $M_o$ of $M_z$ to these subspaces are skew symmetric,
i.e. $M_e + M_e^\ast = 0,$ respectively $ \ M_o + M_o^\ast = 0$. Hence $M_o$ and $M_e$ are normal with purely imaginary spectra.

In conclusion, there are quadrature formulas of the form
$$ \int_{-a}^a [p(x)\overline{q(x)} + p'(x)\overline{q'(x)}] dx = \sum_{k=1}^N c_k p(i a_k) \ol{q(ia_k)},$$
with $c_k >0$ and $a_k \in \R$ for all $k, 1 \leq k \leq N$, valid for both polynomials $p,q$ either of the form
$ (z^2-a^2) f(z^2)$, or $(z^2-a^2) z f(z^2)$ and $f \in \C_n[z]$ with a prescribed $n$. Counting degrees we can specify $N \leq n+1$.
The complex location of the quadrature nodes is another strong indication that passing from real to complex variables is very natural in this case.

\end{example}

\begin{example} To prove that the minimal node number (producing a so called Gaussian type quadrature) in Theorem \ref{real-pos} is needed, we simply look at a counting measure supported by the
vertices of a regular $n$-gon.

Indeed, let $n>2$ be a positive integer and let $\epsilon = e^{i 2\pi/n}$ be the primitive root of order $n$ of unity. The measure
$d\nu = \frac{1}{n}(\delta_{1} + \delta_{\epsilon}+ \ldots + \delta_{\epsilon^{n-1}})$ has complex moments concentrated on the diagonal, up to degree $n$:
$$ \int z^k \ol{z}^\ell d\nu = 0, \ \ 0 \leq k,\ell \leq n, \ k \neq \ell.$$
The associated complex orthonormal polynomials are ${\bf 1}, z, \ldots, z^n$. Let $d<n$ and consider the orthogonal projection $\pi_d$ of the finite Hilbert space
$L^2(\nu)$ onto the space $A_d$ of complex polynomials of degree less than or equal to $d$. The compressed matrix $M_d = \pi_d M_z \pi_d$ of the multiplier $M_z$
is the Jordan block of size $(d+1)\times (d+1)$, hence far from being normal. However, a finite point quadrature exists (with more than $d+1$ nodes)
$$ \langle p(M_d){\bf 1}, q(M_d){\bf 1} \rangle = \int p \ol{q} d\nu,$$
whenever $\deg(p) \leq d+1$ and $\deg(q) \leq d$.

\end{example}

\begin{example}

A non-negative real functional which is positive on hermitian squares is offered by the integration with respect to arc length on the unit circle:
$$ L( z^j \ol{z}^k) = \delta_{jk}, \ \ j,k \geq 0.$$
The associated Hessenberg matrix of order $(d+1) \times (d+1)$ is the Jordan block, that is the Toeplitz matrix with equal entries to $1$ on the sub diagonal 
and zero elsewhere. The complex orthogonal polynomials are $P_j(z)=z^j, \ \ j \geq 0.$

Note that
$$ L(|z|^2 q(z,\ol{z})) = L( q(z,\ol{z})) $$
for any polynomial $q \in \C[z, \ol{z}]$
and in particular 
$$ L(|1-\ol{z} z|^2) =0.$$
In other terms the form $\sigma_M$ and the self-commutator $[M^\ast,M]$ may be degenerate, in spite of the fact that $L$ is positive definite on hermitian squares, that is:
$$ L(|f(z)|^2) = 0 \ \ \Rightarrow f =0.$$
\end{example}

\end{document}